\documentclass[11pt, a4paper]{amsart}

\usepackage{amsfonts,amsmath,amssymb, amscd,fullpage}
\usepackage{stmaryrd}
\usepackage[all]{xy}

\newtheorem{theorem}{Theorem}[section]
\newtheorem{lemma}[theorem]{Lemma}
\newtheorem{proposition}[theorem]{Proposition}
\newtheorem{corollary}[theorem]{Corollary}

\theoremstyle{definition}
\newtheorem{definition}[theorem]{Definition}

\theoremstyle{remark}
\newtheorem{rem}[theorem]{Remark}
\newtheorem{rems}[theorem]{Remarks}
\newtheorem{example}[theorem]{Example}
\newtheorem{examples}[theorem]{Examples}

\newcommand\pf{\begin{proof}}
\newcommand\epf{\end{proof}}

\newcommand\C{\mathbb{C}}

\newcommand\tor{\mathrm{Tor}}
\newcommand\ext{\mathrm{Ext}}

\DeclareMathOperator{\Hom}{Hom}

\DeclareMathOperator{\id}{id}

\numberwithin{equation}{section}

\title{Homological properties of quantum permutation algebras}

\author{Julien Bichon}
\address{
Laboratoire de Math\'ematiques Blaise Pascal,
Universit\'e Clermont Auvergne,
Complexe universitaire des C\'ezeaux,
3 place Vasar\'ely
63178~Aubi\`ere Cedex, France}
\email{julien.bichon@uca.fr}

\author{Uwe Franz}
\address{
Laboratoire de math\'ematiques de Besan\c{c}on,
UMR 6623, CNRS, Universit\'e Bourgogne Franche-Comt\'e,
16 route de Gray,
25030 Besan\c{c}on cedex, France}
\email{uwe.franz@univ-fcomte.fr}

\author{Malte Gerhold}
\address{
Institut f\"ur Mathematik und Informatik,
Walther-Rathenau-Str. 47,
17489 Greifswald, Germany}
\email{malte.gerhold@uni-greifswald.de}

\thanks{This work  was supported by the French ``Investissements d'Avenir'' program, project ISITE-BFC (contract ANR-15-IDEX-03). We acknowledge support by MAEDI/MENESR and DAAD through the PROCOPE programme.}


\begin{document}

\begin{abstract}
We show that $A_s(n)$, the coordinate algebra of Wang's quantum permutation group, is Calabi-Yau of dimension $3$
when $n\geq 4$, and compute its Hochschild cohomology with trivial coefficients.
We also show that, for a larger class of quantum permutation algebras, including those representing quantum symmetry groups of finite graphs,  the second Hochschild cohomology group with trivial coefficients  vanishes, and hence these algebras have the AC property considered in quantum probability: all 
cocycles can be completed to a Sch\"urmann triple.
\end{abstract}

\maketitle

\section{Introduction}

Let $n\geq 1$, and let $A_s(n)$ be the universal quantum permutation algebra: this means that 
$A_s(n)$ is the algebra presented by generators $u_{ij}$, $1\leq i,j \leq n$, subject to the relations of permutation matrices
$$u_{ij}u_{ik}=\delta_{jk}u_{ij}, \ u_{ji}u_{ki}=\delta_{jk}u_{ji}, \ \sum_{j=1}^nu_{ij}=1=\sum_{j=1}^nu_{ji}$$
It  has a natural Hopf algebra structure given by
$$\Delta(u_{ij})
= \sum_{k=1}^n u_{ik} \otimes u_{kj}, \ 
\varepsilon(u_{ij}) = \delta_{ij}, \ 
S(u_{ij}) = u_{ji}$$
The Hopf algebra $A_s(n)$ arose in Wang's work on quantum group actions on finite-dimensional algebras \cite{wan98}, and represents the quantum permutation group $S_n^+$, the largest compact quantum group acting on a set consisting of $n$ points \cite{wan98}. More generally, the Hopf algebra $A_s(n)$ is the universal cosemisimple Hopf algebra coacting on the algebra $\mathbb C^n$ \cite{bi08}. 

The algebra $A_s(n)$ is infinite-dimensional if $n\geq 4$ \cite{wan98}, so that the quantum permutation group $S_n^+$ is infinite in that case, although it acts faithfully on a classical finite set. This intriguing property has brought a lot of attention on $A_s(n)$ among researchers in quantum group theory in the last twenty years, and a number of important contributions elucidate its structure: representation theory \cite{ba99}, operator-theoretic approximation properties \cite{bra}, operator $K$-theory \cite{voi}. On the homological algebra side, the cohomological dimension of $A_s(n)$ has been computed recently in \cite{bic}. In this paper we go deeper in the homological study, proving the following result (the cohomology  in the statement being Hochschild cohomology with trivial coefficients).

\begin{theorem}\label{thm:asncala}
Assume that $n \geq 4$. The algebra $A_s(n)$ is Calabi-Yau of dimension $3$, and we have
$$H^p(A_s(n), \mathbb C) \simeq \begin{cases}
                    \mathbb C  & \text{if $p=0,3$} \\
 0 &\text{otherwise} 
                   \end{cases}
$$ 
\end{theorem}

The Calabi-Yau property, which has been the subject of numerous investigations in recent years, was named in \cite{gin}, and  is the algebraic generalization of the notion of orientable Poincar\'e duality group (\cite{biec}, see \cite[VIII.10]{brown}). One of its interests is that it provides a duality between Hochschild homology and cohomology \cite{vdb}, similar to Poincar\'e duality in algebraic topology.

In fact we will prove (Theorem \ref{thm:Aaut}) that the quantum symmetry algebra of a finite-dimensional semisimple normalizable measured algebra of dimension $\geq 4$ is always twisted Calabi-Yau of dimension $3$, see Section 4 for details. To prove this result, our main tools will be a new result that gives a condition that ensures that a Hopf subalgebra of a twisted Calabi-Yau algebra of dimension $d$ still is a twisted Calabi-Yau of dimension $d$ (see Theorem \ref{thm:subcala} in Section 3), techniques from recent papers of the first author \cite{bic13,bic}, and crucial results of Mrozinski \cite{mro} and of Wang-Yu-Zhang \cite{wyz}.

Notice that Theorem \ref{thm:asncala} states in particular that the second cohomology group $H^2(A_s(n), \mathbb C)$ vanishes, and hence in particular $A_s(n)$ has the property called AC in \cite{fgt}. The AC property is of particular interest in quantum probability and the study of L\'evy processes on quantum groups: it means that all cocycles can be completed to a Sch\"urmann triple. See \cite{fgt} for details, and the recent survey \cite{fks} on these questions, where the AC property was shown for $A_s(n)$.
In fact we will show that the vanishing result for the second cohomology (and hence the AC property) holds for a large class of quantum permutation algebras (quotients of $A_s(n)$), including those representing quantum symmetry groups of finite graphs, in the sense of \cite{ban}: see Theorem \ref{thm:vanish}.

The paper is organized as follows. Section 2 consists of preliminaries. In Section $3$ we discuss 
Hopf subalgebras of Hopf algebras that are twisted Calabi-Yau. In Section $4$ we prove our results regarding the Calabi-Yau  property for quantum symmetry algebras and Theorem \ref{thm:asncala}. The final Section 5 is devoted to prove the vanishing of the second cohomology (with trivial coefficients) of a large class of quantum permutation algebras.

\section{Preliminaries}

\subsection{Notations and conventions}
We work over $\mathbb C$ (or over any algebraically closed field of characteristic zero). 
We assume that the reader is familiar with the theory of Hopf algebras and their tensor categories of comodules, as e.g. in \cite{kas,ks,mon}. 
If $A$ is a Hopf algebra, as usual, $\Delta$, $\varepsilon$ and $S$ stand respectively for the comultiplication, counit and antipode of $A$. As usual, the augmentation ideal ${\rm Ker}(\varepsilon)$ is denoted $A^+$.
We use Sweedler's notations in the standard way. 

The category of right $A$-comodules is denoted $\mathcal M^A$, the category of right $A$-modules is denoted $\mathcal M_A$, etc...  
The trivial right $A$-module is denoted $\C_\varepsilon$, the trivial left $A$-module is denoted ${_\varepsilon \! \mathbb C}$, and the trivial $A$-bimodule ${_\varepsilon \!\C}_\varepsilon$ is simply denoted  $\mathbb C$.
The set of $A$-module morphisms (resp. $A$-comodule morphisms) between two $A$-modules (resp. two $A$-comodules) $V$ and $W$ is denoted ${\rm Hom}_A(V,W)$ (resp. ${\rm Hom}^A(V,W)$).
We will mostly consider right $A$-modules, and the corresponding ${\rm Ext}$-spaces, denoted ${\rm Ext}_A^*(-,-)$, will always be considered in the category of right $A$-modules

We also assume some familiarity with homological algebra \cite{wei}, but for the reader's convenience, we recall the definition of Hochschild (co)homology:
if $A$ is an algebra and $M$ is an $A$-bimodule, the Hochschild cohomology
spaces $H^*(A,M)$ are the cohomology spaces of the complex
$$ 0 \longrightarrow {\rm Hom}(\C,M) \overset{\delta}\longrightarrow \Hom(A,M) \overset{\delta}\longrightarrow
\cdots \overset{\delta}\longrightarrow\Hom(A^{\otimes n}, M) \overset{\delta}\longrightarrow \Hom(A^{\otimes n+1}, M) \overset{\delta}\longrightarrow \cdots$$
where the differential $\delta \colon \Hom(A^{\otimes n}, M) \longrightarrow \Hom(A^{\otimes n+1}, M)$
is given by 
\begin{align*}
\delta(f)(a_1 \otimes \cdots \otimes a_{n+1}) = &a_1\cdot
f(a_2 \otimes \cdots \otimes a_{n+1}) + \sum_{i=1}^{n}(-1)^i f(a_1 \otimes \cdots \otimes a_i a_{i+1} \otimes \cdots \otimes a_{n+1}) \\
&+ (-1)^{n+1} f(a_1 \otimes \cdots \otimes a_{n}) \cdot a_{n+1}
\end{align*}
The Hochschild cohomological dimension of $A$, which serves as a noncommutative analogue of the dimension of an algebraic variety, is defined by
\begin{align*}
{\rm cd}(A)&= {\rm sup}\{n : \ H^n(A, M) \not=0 \ \text{for some $A$-bimodule $M$}\}\in \mathbb N \cup \{\infty\} \\
& = {\rm min}\{n : \ H^{n+1}(A, M) =0  \ \text{for any  $A$-bimodule} \ M\}  \\
&={\rm pd}_{_A\mathcal M_A}(A) 
\end{align*}
where ${\rm pd}_{_A\mathcal M_A}(A)$ is the projective dimension of $A$ in the category of $A$-bimodules.

The Hochschild homology
spaces $H_*(A,M)$ are the homology spaces of the complex
$$\cdots \longrightarrow  M \otimes A^{\otimes n}  \overset{b}\longrightarrow M \otimes A^{\otimes n-1} \overset{b}\longrightarrow \cdots 
\overset{b}\longrightarrow M \otimes A \overset{b}\longrightarrow M \longrightarrow 0 \quad $$
where the differential $b \colon  M \otimes A^{\otimes n}  \longrightarrow M \otimes A^{\otimes n-1}$ is given by
\begin{align*}
b(x\otimes a_1 \otimes \cdots \otimes a_n  )
=&x \cdot a_1 \otimes \cdots \otimes a_{n} + 
 \sum_{i=1}^{n-1}(-1)^i x \otimes a_1 \otimes \cdots \otimes a_i a_{i+1} \otimes \cdots \otimes a_{n} \\
&+ (-1)^n a_n\cdot x \otimes a_1 \otimes \cdots \otimes a_{n-1}
\end{align*}

As said in the introduction, the Calabi-Yau property is a condition that ensures a Poincar\'e type duality between Hochschild cohomology and homology. We will not give the definition here, since we will consider this property for Hopf algebras only, and that in this framework there is a simpler and more useful one (at least to us), to be given in the next subsection.

\subsection{Hochschild cohomology, Hopf algebras, and the Calabi-Yau property} Let us first recall the well-known fact that for a Hopf algebra, the previous Hochschild (co)homology spaces can described by using suitable ${\rm Ext}$ and ${\rm Tor}$ on the category of left or right modules. Indeed, let $A$ be Hopf algebra and let $M$ be an $A$-bimodule. We then have
$$H_*(A,M) \simeq \tor^A_*(\C_\varepsilon, M'), \quad H^*(A,M) \simeq \ext^*_A(\C_\varepsilon, M'')$$
where $M'$ is $M$ having the left $A$-module structure defined by $ a \rightarrow x = a_{(2)}\cdot x \cdot S(a_{(1)})$, and $M^{''}$ is $M$ having the right $A$-module structure defined by
 $
 x \leftarrow a = S(a_{(1)})\cdot x \cdot a_{(2)}$. 
See for example \cite{bic13,bz}, and the references therein.

%The above $\ext$-spaces are those in the category of right $A$-modules (and in this paper, all $\ext$-spaces will be in the category of right modules).
 
\begin{definition}
A Hopf algebra $A$ is said to be \textsl{twisted Calabi-Yau of dimension  $d\geq 0$} if it satisfies the following conditions.
\begin{enumerate}
 \item $A$ is homologically smooth, i.e.\ the trivial $A$-module $\mathbb C_\varepsilon$ admits a finite resolution by finitely generated projective $A$-modules.
\item ${\rm Ext}_A^i({\mathbb C_\varepsilon, A_A})=0$ if $i \not=d$, and ${\rm Ext}_A^d(\mathbb C_\varepsilon, A_A)$ is one-dimensional, so that there exists an algebra map $\alpha : A \to \mathbb C$ such that ${\rm Ext}_A^d(\mathbb C_\varepsilon, A_A)\simeq {_\alpha\!\mathbb C}$ as left $A$-modules. The algebra is said to be Calabi-Yau if $\alpha=\varepsilon$.
\end{enumerate} 
\end{definition}

The equivalence of this definition with the usual one for twisted Calabi-Yau algebras has been discussed in several papers, we refer the reader to \cite{wyz}. Notice that if  $A$ is  twisted Calabi-Yau of dimension  $d$, then $d={\rm cd}(A)$.

The induced homological duality is as follows, see e.g. \cite{bz}. 
If ${\rm Ext}^d_A(\mathbb C , A)\simeq {_\alpha\!\mathbb C}$, then consider the algebra anti-morphism $\theta \colon A \rightarrow A$ defined by  $\theta(a)=S(a_{(1)})\alpha(a_{(2)})$. We then have, for any right $A$-module $M$,
$${\rm Ext}_A^i(\mathbb C_\varepsilon, M) \simeq {\rm Tor}_{d-i}^A(\mathbb C_\varepsilon, {_{\theta}\!M})$$ 
where ${_{\theta}\!M}$ is $M$ having the left $A$-module structure given by $a\cdot x=x\cdot\theta(a)$. At the level of Hochschild cohomology, letting $\sigma = S\theta$, we have, for any $A$-bimodule $M$,
$$H^i(A,M) \simeq H_{d-i}(A, {_{\sigma}\!M})$$ 
where ${_{\sigma}\!M}$ has the $A$-bimodule structure given by $a\cdot x \cdot b= \sigma(a)\cdot x\cdot b$.

The following easy lemma will be used in Section 4.

\begin{lemma}\label{lemm:tcy-cy}
 Let $A$ be Hopf algebra, and assume that $A$ is twisted Calabi-Yau of dimension $d$. If ${\rm Ext}_A^d(\mathbb C_\varepsilon, \mathbb C_\varepsilon) \not=(0)$, then $A$ is Calabi-Yau. 
\end{lemma}

\begin{proof}
Let $\alpha \colon A \rightarrow \mathbb C$ be the algebra map such that ${\rm Ext}^d_A(\mathbb C_\varepsilon , A_A)\simeq {_\alpha\!\mathbb C}$. By the previous homological duality, we have
$${\rm Ext}^d_A(\mathbb C_\varepsilon, \mathbb C_\varepsilon) \simeq {\rm Tor}_0^A(\mathbb C_\varepsilon, {_\alpha \!\mathbb C})
\simeq \begin{cases}
  0 & \text{if $\alpha \not= \varepsilon$} \\
\mathbb C & \text{if $\alpha = \varepsilon$}
 \end{cases}
$$
Hence the assumption ensures that $\alpha = \varepsilon$, and $A$ is Calabi-Yau, as claimed.
\end{proof}

As examples of Hopf algebras that are twisted Calabi-Yau, let us mention, for example, the group algebras of Poincar\'e duality groups \cite{brown}, the coordinate algebras of the $q$-deformed quantum groups \cite{bz}, the coordinate algebras of the free orthogonal quantum groups \cite{cht}, and more generally the coordinate algebras of quantum symmetry groups of nondegenerate bilinear forms \cite{bic13,walwan,yu}.

\subsection{Exact sequences of Hopf algebras}
A sequence  of Hopf algebra maps
\begin{equation*}\C \to B \overset{i}\to A \overset{p}\to L \to
\C\end{equation*} is said to be exact \cite{ad} if the following
conditions hold:
\begin{enumerate}\item $i$ is injective and $p$ is surjective,
\item ${\rm Ker}(p) =Ai(B)^+ =i(B)^+A$, where $i(B)^+=i(B)\cap{\rm Ker}(\varepsilon)$,
\item $i(B) = A^{{\rm co} L} = \{ a \in A:\, ({\id} \otimes p)\Delta(a) = a \otimes 1
\} = {^{{\rm co} L}A} = \{ a \in A:\, (p \otimes {\id})\Delta(a) = 1 \otimes a
\}$. \end{enumerate}
Note that condition (2) implies  $pi= \varepsilon 1$. We do not lose generality in assuming that $B \subset A$ is a Hopf subalgebra and that $i$ is the inclusion map.

An exact sequence as above and such that $A$ is faithfully flat as a right $B$-module is called strict \cite{schneider}. If $L$ is cosemisimple, then an exact sequence is automatically strict, by \cite[Theorem 2]{tak}, since $p\colon A \rightarrow L$ is then faithfully coflat.

The example we are interested in is the following one. Let $\mathcal O({\rm SL}_q(2))$ be the coordinate algebra on quantum ${\rm SL}(2)$, with its standard generators $a,b,c,d$. Let $\mathcal O({\rm PSL}_q(2))$ be the subalgebra generated by the elements $xy$, $x,y \in\{a,b,c,d\}$, and let $p \colon \mathcal O({\rm SL}_q(2)) \rightarrow \mathbb C \mathbb Z_2$ be the Hopf algebra map defined by $\pi(a)=\pi(d)=g$ and $\pi(b)=\pi(c)=0$, where $g$ is the generator of the the cyclic group $\mathbb Z_2$. Then it is a direct verification that the sequence   
\begin{equation*}\C \to  \mathcal O({\rm PSL}_q(2)) \to \mathcal O({\rm SL}_q(2)) \overset{p}\to \mathbb C \mathbb Z_2 \to
\C\end{equation*}
is exact. This exact sequence is strict since $\mathbb C \mathbb Z_2$ is cosemisimple.

\section{Hopf subalgebras and the Calabi-Yau property}

Let $B \subset A$ be a Hopf subalgebra. The aim of this section is to provide a condition that ensures that if $A$ is twisted Calabi-Yau of dimension $d$, then so is $B$.
Our result is the following one.

\begin{theorem}\label{thm:subcala}
 Let $\mathbb C \to B \to A \to L \to \mathbb C$ be a strict exact sequence of Hopf algebras with bijective antipodes, with $L$  finite-dimensional. The following assertions are equivalent.
\begin{enumerate}
\item $A$ is a twisted Calabi-Yau algebra of dimension $d$.
\item ${\rm cd}(A)$ is finite and $B$ is a twisted Calabi-Yau algebra of dimension $d$. 
\end{enumerate}
\end{theorem}

The proof will be  given at the end of the section. The proof of $(1)\Rightarrow (2)$  will be a consequence of two results of independent interest (Propositions \ref{prop:subsmooth} and \ref{prop:subext}), while the proof of $(2)\Rightarrow (1)$ will follow from Proposition \ref{prop:subext} together with some considerations on resolutions by finitely generated projective modules from \cite[Chapter VIII]{brown}. 

\subsection{Relative integrals}
We will use the following notion.

\begin{definition}
 Let $B \subset A$ be a Hopf subalgebra. 
%\begin{enumerate}
\textsl{ A right $A/B$-integral} is an element $t \in  A$ such that  
 for any $a \in A$, we have
$ta -\varepsilon(a)t \in B^+A$. A right $A/B$-integral $t$ is said to be non-trivial if $t \not\in B^+A$.
%\item A right $A/B$-integral is a right $B$-linear map $h : A/B^+A \rightarrow \mathbb C$ such that for any $a \in A$, we have $hp(a_{(2)})p(a_{(1)})=hp(a)p(1)$.
%\end{enumerate}
\end{definition}

\begin{examples}\label{exam:integ}
\begin{enumerate}
 \item 
 Let $A$ be a Hopf algebra. A right $A/\mathbb C$-integral is a right integral in $A$, in the usual sense \cite{mon, rad}.
\item Let  $B \subset A$ be a Hopf subalgebra such that $B^+A = AB^+$, so that $L=A/B^+A$ is a quotient Hopf algebra. If $\tau \in L$ is a right integral, an element $t \in A$ such that $p(t)=\tau$ (where $p \colon A \rightarrow L$ is the canonical surjection) 
 is a right $A/B$-integral. 
\item Let $G$ be a discrete group,  let $H \subset G$ be a subgroup of finite index and $g_1, \ldots , g_n \in G$ be such that $G= \cup_{i=1}^nHg_i$ (disjoint union). Then $t=\sum_{i=1}^n g_i$ is a right $\mathbb C G/ \mathbb CH$-integral. 
\end{enumerate}
\end{examples}

We will discuss, at the end of the section, an existence result for an $A/B$-integral, beyond the basic examples above.

We begin with a lemma, which illustrates the role of such an $A/B$-integral.

\begin{lemma}\label{lem:ave}
 Let $B \subset A$ be a Hopf subalgebra, and let $t \in A$ be a right $A/B$-integral.
Let $M,N$ be right $A$-modules, and let $f \colon M \rightarrow N$ be a right $B$-linear map.
\begin{enumerate}
\item The map $\tilde{f} \colon M \rightarrow N$, defined by $\tilde{f}(x)=f(x\cdot S(t_{(1)}))\cdot t_{(2)}$, is $A$-linear.
\item  If there exists an $A$-linear map $i \colon N \rightarrow M$ such that $fi={\rm id}_N$, then $\tilde{f}i=\varepsilon(t){\rm id}_N$.
\end{enumerate}
\end{lemma}

\begin{proof}
 If $M$ is an $A$-module, we denote, as usual, by $M^A$ the space of $A$-invariants: 
$$M^A=\{x \in M \ | \ x\cdot a=\varepsilon(a)x, \ \forall a \in A\}$$ It is easy to check that $M^B\cdot t \subset M^A$ (this is an equality if $\varepsilon(t) \not=0$).

Now if $M$ and $N$ are (right) $A$-modules,  recall that ${\rm Hom}(M,N)$ admits a right $A$-module structure defined by 
$$f\cdot a(x)= f(x\cdot S(a_{(1)}))\cdot a_{(2)}$$
and that 
${\rm Hom}_A(M,N)= {\rm Hom}(M,N)^A$. Starting with $f \in {\rm Hom}_B(M,N)$, the previous remark that shows that $\tilde{f}=f\cdot t \in {\rm Hom}_A(M,N)$, as announced. The last statement is a direct verification.
\end{proof}

We now will need a certain category of relative Hopf modules, considered in \cite{tak}, see \cite{sch90} as well.  Let $B \subset A$ be a Hopf subalgebra. Then $B^+A$ is co-ideal in $A$, so that $A/B^+A$ is coalgebra and the quotient map $$p \colon A \rightarrow A/B^+A$$ is a coalgebra map, and a right $A$-module map. We denote by $\mathcal M_A^C$ the category whose objects are right $C$-comodules and right $A$-modules $V$ satisfying the following compatibility condition
$$(v\cdot a)_{(0)} \otimes (v\cdot a)_{(1)}= v_{(0)}\cdot a_{(1)} \otimes v_{(1)}\cdot a_{(2)}, \quad v \in V, \ a \in A$$ 
The morphisms are the $C$-colinear and $A$-linear maps. As examples, $A$ and $C$, endowed with the obvious structures, are objects of $\mathcal M_A^C$. Notice that if $V$ is an object of $\mathcal M_A^C$, then $$V^{{\rm co} C}=\{v \in V \ | \ v_{(0)} \otimes v_{(1)} = v \otimes p(1)\}$$ is a sub-$B$-module of $V$, since for $b \in B$, we have $p(b)=\varepsilon(b)p(1)$.

\begin{proposition}\label{prop:subsmooth}
 Let $B \subset A$ be a Hopf subalgebra. Assume that $A$  has bijective antipode, that $A$ is faithfully flat as a left or right $B$-module, and that there exists a non-trivial right $A/B$-integral $t \in A$. 
\begin{enumerate}
 \item 
The coalgebra $C=A/B^+A$ is finite dimensional and $A$ is finitely generated and projective as a right and left $B$-module.
\item If $A$ is homologically smooth, then so is $B$, with ${\rm cd}(B)\leq {\rm cd}(A)$.
\item If moreover $\varepsilon(t)\not=0$, then ${\rm cd}(A)={\rm cd}(B)$.
\end{enumerate}

\end{proposition}

\begin{proof}
 First notice that by \cite[Corollary 1.8]{schn}, $A$ is indeed projective as a left and right $B$-module.
Let $t \in A$ be a non-trivial $A/B$-integral, and let $V$ be the right subcomodule of $C$ generated by $p(t)$:
$$V= {\rm Span}\{p(t_{(1)})\psi p(t_{(2)}), \ \psi \in C^*\}$$ 
For $a \in A$, we have 
\begin{align*}
& p(ta) =\varepsilon(a)p(t) \Rightarrow p(ta_{(1)}) \otimes a_{(2)}=p(t) \otimes a \\ \Rightarrow & \ p(t_{(1)}a_{(1)}) \otimes p(t_{(2)}a_{(2)})\otimes a_{(3)}=p(t_{(1)}) \otimes p(t_{(2)}) \otimes a \\ 
\Rightarrow & \ p(t_{(1)}a_{(1)}) \otimes p(t_{(2)}a_{(2)})\cdot S(a_{(3)})= p(t_{(1)}) \otimes p(t_{(2)})\cdot S(a) \\
\Rightarrow & \ p(t_{(1)}a) \otimes p(t_{(2)})= p(t_{(1)}) \otimes p(t_{(2)}\cdot S(a))\\
\end{align*}
Hence for $\psi \in C^*$, we have $\psi p(t_{(2)})p(t_{(1)})\cdot a = \psi p(t_{(2)})p(t_{(1)} a)=\psi p(t_{(2)}\cdot S(a))p(t_{(1)})$, and this shows that $V$ is a sub-$A$-module of $C$. Hence $V$ is an object of $\mathcal M_A^C$,
and by \cite[Lemma 1.3]{schn},  we can use Theorem 3.7 and its proof in \cite{sch90} to claim that the natural map
$$V^{{\rm co} C} \otimes_B A \rightarrow V, \ v \otimes_B a \mapsto v\cdot a$$ 
is an isomorphism (we could also deduce this from left-right variations on \cite[Theorem 1, Theorem 2]{tak}). Since $C^{{\rm co} C}$ is one dimensional, generated by $p(1)$, we thus see that $V^{{\rm co} C}$ contains $p(1)$ (since $V$ is non-zero), and hence $V=C$. Since $V$ is finite-dimensional, we conclude that $C$ is.

So let $a_1, \ldots , a_n\in A$ be elements in $A$ such that $p(a_1), \ldots, p(a_n)$ linearly generate $C$, let $W$ be the (finite-dimensional) right $A$-subcomodule of $A$ generated by    $a_1, \ldots , a_n$, and let $BW$ be the left $B$-submodule generated by $W$. It is clear that $BW$ is still a subcomodule, so is an object in the category of Hopf modules ${_B\!\mathcal M}^A$. By \cite[Theorem 1]{tak} and \cite[Remark 1.3]{musc}, the functor 
$${_B\!\mathcal M}^A \longrightarrow \mathcal M^C, \ M \longmapsto M/B^+M$$
is a category equivalence. Since $BW$ and $A$ have the same image by this functor (the object $C$), and $BW\subset A$, we have $BW=A$ and $A$ is finitely generated as a left $B$-module, as well as a right $B$-module, by bijectivity of the antipode. This concludes the proof of (1), and (2) follows, since the restriction of a finitely generated projective $A$-module to a $B$-module is still finitely generated and projective, and then ${\rm cd}(B) \leq {\rm cd}(A)$.

To prove the converse inequality under the assumption that $\varepsilon(t)\not=0$, we can assume that $m={\rm cd}(B)$ is finite. Consider a resolution of the trivial (right) $A$-module
$$\cdots  \rightarrow P_n \rightarrow P_{n-1} \rightarrow \cdots\rightarrow  P_1 \rightarrow P_0 \rightarrow \mathbb C_\varepsilon$$
by projective $A$-modules. These are in particular projective as $B$-modules by (1), so since $m={\rm cd}(A)$, a standard argument yields an exact sequence of $A$-modules
$$0 \rightarrow K \overset{i}\rightarrow P_m \rightarrow P_{m-1} \rightarrow \cdots\rightarrow  P_1  \rightarrow P_0 \rightarrow \mathbb C_\varepsilon$$
together with $r \colon P_m  \rightarrow K$, a $B$-linear map such that $ri = {\rm id}_K$. Lemma \ref{lem:ave} yields an $A$-linear map $\tilde{r} \colon P_m \rightarrow K$ such that $\tilde{r}i=\varepsilon(t){\rm id}_K$.
 We thus obtain, since a direct summand of a projective is projective, a length $m$ resolution of $\mathbb C_\varepsilon$ by projective  $A$-modules, and we conclude that ${\rm cd}(A)\leq m$, as required.
\end{proof}

\begin{proposition}\label{prop:subext}
 Let  $B \subset A$ be a Hopf subalgebra. Assume that the antipode of $A$ is bijective, that  $A$ is faithfully flat as a left or right $B$-module, that $B^+A = AB^+$ (so that $L=A/B^+A$ is a quotient Hopf algebra), and 
 that $L$ is finite-dimensional. Then we have, for any $M \in \mathcal M_A$ and any $N \in \mathcal M_A^{L}$, 
$${\rm Ext}_A^*(M,N) \simeq {\rm Ext}_B^*(M_{|B}, N^{{\rm co} L})$$
\end{proposition}

\begin{proof}
Since $L$ is a finite-dimensional Hopf algebra, there exists a right integral $\tau \in L$ and 
a left integral  $h \colon L \rightarrow \mathbb C$ on $L$, 
such that $h(\tau)=1$ and $hS(\tau)\not=0$ (see e.g. \cite[Theorem 10.2.2]{rad}).

If $f \in {\rm Hom}_B(M_{|B}, N^{{\rm co} L})$, we can view $f$ as a $B$-linear map $M \rightarrow N$, so using an element $t \in A$ such that $p(t)=\tau$, Lemma \ref{lem:ave} gives a linear map 
\begin{align*}
 \Psi \colon  {\rm Hom}_B(M_{|B}, N^{{\rm co} L}) & \longrightarrow {\rm Hom}_A(M,N) \\
f & \longmapsto \tilde{f}, \ \tilde{f}(x)=f(x\cdot S(t_{(1)}))\cdot t_{(2)}
\end{align*}
To construct the inverse, notice that we have a map
\begin{align*}
 E_N \colon N & \longrightarrow N^{{\rm co} L} \\
x & \longmapsto h(x_{(1)}) x_{(0)}
\end{align*}
which is $B$-linear since $h$ is. We therefore have a linear map
\begin{align*}
 \Phi \colon {\rm Hom}_A(M,N) & \longrightarrow {\rm Hom}_B(M_{|B}, N^{{\rm co} L}) \\
g &\longmapsto E_N g
\end{align*}
For $f \in {\rm Hom}_B(M_{|B}, N^{{\rm co} L})$ and $x \in M$, we have 
\begin{align*}
 E_N\tilde{f}(x) & = E_N(f(x\cdot S(t_{(1)}))\cdot t_{(2)}) \\
& = h\left((f(x\cdot S(t_{(1)}))\cdot t_{(2)})_{(1)}\right) (f(x\cdot S(t_{(1)}))\cdot t_{(2)})_{(0)} \\
& = h\left(f(x\cdot S(t_{(1)}))_{(1)}\cdot t_{(3)}\right) f(x\cdot S(t_{(1)}))_{(0)}\cdot t_{(2)} \\
& = h(p(1)\cdot t_{(3)}) f(x\cdot S(t_{(1)}))\cdot t_{(2)} \quad ({\rm since} \ f(M)\subset N^{{\rm co} L})\\
& = f(x\cdot S(t_{(1)}))\cdot t_{(2)} hp(t_{(3)})
\end{align*}
The last term on the right belongs to $A^{{\rm co} L}=B$ (\cite[Lemma 1.3]{schn}, or \cite[Theorem 1]{tak}), so the $B$-linearity of $f$ gives 
\begin{align*}
  E_N\tilde{f}(x) & = f(x\cdot S(t_{(1)})\cdot t_{(2)}) hp(t_{(3)}) \\
& = f(x) h(p(t))= f(x)h(\tau)=f(x)
\end{align*}
Hence $ E_N\tilde{f}=f$, and $\Phi\Psi={\rm id}$. Now let $g \in {\rm Hom}_A(M,N)$ and $x \in M$. We have
\begin{align*}
 \widetilde{E_Ng}(x)&=E_Ng(x\cdot S(t_{(1)}))\cdot t_{(2)} = E_N(g(x)\cdot S(t_{(1)}))\cdot t_{(2)} \\
&=h(g(x)_{(1)} \cdot S(t_{(1)}))g(x)_{(0)}\cdot S(t_{(2)})t_{(3)} =
h(g(x)_{(1)}\cdot S(t))g(x)_{(0)}
\end{align*}
For $a \in A$, we have $p(ta)=\varepsilon(a)p(t)$, so $p(a)\cdot S(t)= p(aS(t))=pS(tS^{-1}(a))= \varepsilon(a)pS(t)=\varepsilon(a)Sp(t)$. Hence we have 
$$ \widetilde{E_Ng}(x)= hSp(t)g(x)$$
and $\Psi\Phi=hSp(t){\rm id}$. Combined with the already established identity $\Phi\Psi={\rm id}$, this shows that $\Phi$ and $\Psi$ are inverse isomorphisms. 

The isomorphism $\Psi \colon {\rm Hom}_B(M_{|B}, N^{{\rm co} L}) \rightarrow {\rm Hom}_A(M,N)$ is easily seen to be functorial in $M$, so we can finish the proof by using standard arguments. Let $P_. \rightarrow M$ be a projective resolution by $A$-modules. Since $A$ is projective as a $B$-module (by the previous proposition), the resolution $P_{.|B} \rightarrow M_{|B}$ is by projective $B$-modules, and the isomorphism $\Psi$ induces isomorphisms of complexes
$$H^*({\rm Hom}_B(P_{.|B}, N^{{\rm co} L})) \simeq H^*({\rm Hom}_A(P_.,N))$$
with the term on the left isomorphic to ${\rm Ext}_B^*(M_{|B}, N^{{\rm co} L})$, and the one on the right isomorphic to ${\rm Ext}_A^*(M,N)$. This concludes the proof. 
\end{proof}

\subsection{Finiteness conditions for projective resolutions} We now discuss a convenient reformulation of homological smoothness. This is probably well-known, and is an adaptation of the case of groups  \cite{brown}, but in lack of suitable reference, we provide some details.
The following definition is from \cite[Chapter VIII]{brown}.

\begin{definition}
 Let $A$ be an algebra and let $M$ be an $A$-module.
\begin{enumerate}
 \item The $A$-module $M$ is said to be of \textsl{type $FP_n$}, for $n\geq 0$,  if there exists a  partial projective resolution
$$P_n \rightarrow P_{n-1} \rightarrow \cdots \rightarrow P_1 \rightarrow P_0 \rightarrow M \rightarrow 0$$
with each $P_i$ finitely generated as an $A$-module.
\item The $A$-module $M$ is said to be of \textsl{type $FP_\infty$} if it is of type $FP_n$, for any $n\geq 0$.
\end{enumerate}
If $A$ is a Hopf algebra, then $A$ is said to be of \textsl{type $FP_n$}, for $n\in \mathbb N \cup \{\infty\}$, if the trivial $A$-module $\C_\varepsilon$ is of type $FP_n$.
\end{definition}

We have the following useful characterizations for types $FP_n$ and $FP_\infty$.

\begin{proposition} \label{prop:charFP}
 Let $A$ be an algebra and let $M$ be an $A$-module.
\begin{enumerate}
 \item For $n\geq 0$, the $A$-module $M$ is of type $FP_n$ if and only if $M$ is finitely generated and for every partial projective resolution $$P_k \rightarrow P_{k-1} \rightarrow \cdots \rightarrow P_1 \rightarrow P_0 \rightarrow M \rightarrow 0$$
with $k<n$ and each $P_i$ finitely generated, it holds that ${\rm Ker}(P_k\rightarrow P_{k-1})$ is finitely generated.
\item The $A$-module $M$ is of type $FP_\infty$ if there exists a projective resolution 
$$\cdots \rightarrow P_{n+1} \rightarrow P_n \rightarrow P_{n-1} \rightarrow \cdots \rightarrow P_1 \rightarrow P_0 \rightarrow M \rightarrow 0$$
with each $P_i$ finitely generated.
\end{enumerate}
\end{proposition}

\begin{proof}
 The first assertion is \cite[Proposition 4.3]{brown}, and the second one is \cite[Proposition 4.5]{brown}.
\end{proof}

We arrive at the announced characterization of homological smoothness for Hopf algebras.

\begin{proposition}\label{prop:charsmmoth}
Let $A$ be a  Hopf algebra. Then $A$ is homologically smooth if and only if ${\rm cd}(A)<\infty$ and $A$ is of type $FP_\infty$.
\end{proposition}

\begin{proof}
The direct implication is obvious, while the proof of the converse assertion is the same as the one of \cite[Proposition 6.1]{brown}. Assume that ${\rm cd}(A)=n<\infty$ and that $A$ is of type $FP_\infty$, consider  a projective resolution of $\C_\varepsilon$ as in  (2) of Proposition \ref{prop:charFP}
$$\cdots \rightarrow P_{n+1} \rightarrow P_n \rightarrow P_{n-1} \rightarrow \cdots \rightarrow P_1 \rightarrow P_0 \rightarrow M \rightarrow 0$$
and use (1) of  Proposition \ref{prop:charFP} to get a resolution 
$$0\rightarrow P'_n \rightarrow P_{n-1} \rightarrow \cdots \rightarrow P_1 \rightarrow P_0 \rightarrow \C_\varepsilon \rightarrow 0$$
where each term is finitely generated. Since $n={\rm cd}(A)$, the module $P'_n$ is projective as well, and $A$ is indeed homologically smooth. 
\end{proof}

We conclude the subsection with a last ingredient for the proof of Theorem \ref{thm:subcala}.

\begin{proposition}\label{prop:subFP}
  Let  $B \subset A$ be a Hopf subalgebra with $A$ projective and finitely generated as a right $B$-module. Then  $A$ is of type $FP_n$ if and only if $B$ is, for any $n\in \mathbb N\cup \{\infty\}$.
\end{proposition}
 
\begin{proof}
The restriction of a finitely generated and projective $A$-module to $B$ is again  finitely generated and projective, so it is clear that if $A$ is of type $FP_n$, so is $B$. The proof of the converse assertion is the same as the one in \cite[Proposition 5.1]{brown}. Assume that $B$ is of type $FP_n$, and consider a projective $A$-module partial resolution $$P_k \rightarrow P_{k-1} \rightarrow \cdots \rightarrow P_1 \rightarrow P_0 \rightarrow \C_\varepsilon \rightarrow 0$$ with $k<n$ and each $P_i$ finitely generated as an $A$-module. Then each $P_i$ is finitely generated and projective as a $B$-module, and Proposition \ref{prop:charFP} ensures that ${\rm Ker}(P_k \rightarrow P_{k-1})$ is finitely generated as a $B$-module, and hence as an $A$-module. Proposition \ref{prop:charFP} thus ensures that $A$ is of type $FP_n$.
\end{proof}

\subsection{}
We have now all the ingredients to prove Theorem \ref{thm:subcala}.

\begin{proof}[Proof of Theorem \ref{thm:subcala}]
  Let $\mathbb C \to B \to A \to L \to \mathbb C$ be a strict exact sequence of Hopf algebras with bijective antipodes, with $L$  finite-dimensional. Then, by Example \ref{exam:integ} and Proposition \ref{prop:subsmooth}, $A$ is projective and finitely generated as a left and right $B$-module.

Assume that $A$ is twisted Calabi-Yau of dimension $d$. Then ${\rm cd}(A)=d$ is finite,
and $B$ is homologically smooth by Proposition \ref{prop:subsmooth}, with  ${\rm cd }(B) \leq {\rm cd}(A)=d$. Moreover, Proposition \ref{prop:subext}, applied to $M=\mathbb C_\varepsilon$ and $N=A$, yields 
$${\rm Ext}_A^*(\mathbb C_{\varepsilon},A_A) \simeq {\rm Ext}_B^*(\mathbb C_\varepsilon, A^{{\rm co} L})={\rm Ext}_B^*(\mathbb C_\varepsilon, B_B)$$
so we see simultaneously that ${\rm cd}(B)=d$ and that $B$ is twisted Calabi-Yau of dimension $d$.

Assume now that ${\rm cd}(A)$ is finite and that $B$ is twisted Calabi-Yau of dimension $d$. Then $B$ is of type $FP_\infty$, and by Proposition \ref{prop:subFP}, so is $A$, and $A$ is homologically smooth by Proposition \ref{prop:charsmmoth}. Then, similarly as above, we conclude from Proposition \ref{prop:subext} that $A$ is  twisted Calabi-Yau of dimension $d$.
\end{proof}

Theorem \ref{thm:subcala}, together with the known fact that $\mathcal O({\rm SL}_q(2))$ is twisted Calabi-Yau of dimension $3$ (see \cite{bz}) and the discussion at the end of the previous section, yield the following result.

\begin{corollary}\label{coro:pslCY}
 The algebra $\mathcal O({\rm PSL}_q(2))$ is twisted Calabi-Yau of dimension $3$. 
\end{corollary}

\iffalse

\begin{rem}
A partial converse to Theorem \ref{thm:subcala} holds as well: if $\mathbb C \to B \to A \to L \to \mathbb C$ is a strict exact sequence of Hopf algebras with bijective antipodes, with $L$  finite-dimensional and ${\rm cd}(A)$ finite (which holds if $L$ is semisimple), then  if $B$ is a twisted Calabi-Yau algebra of dimension $d$,  so is $A$.
\end{rem}

\begin{proof}
Since $A$ is finitely generated and projective as a $B$-module, an $A$-module is finitely generated and projective if and only if it is as a $B$-module, and since ${\rm cd}(A)$ is finite, we can argue as in \cite[Chapter VIII]{brown} (see Proposition 5.1 and Proposition 6.1 there)
to conclude that $A$ is homologically smooth  if $B$ is. Then similarly to the proof of Theorem \ref{thm:subcala}, Proposition \ref{prop:subext} shows that $A$ is Calabi-Yau of dimension $d$ if $B$ is, with then ${\rm cd}(A)=d$.
\end{proof}

\fi

\subsection{On the existence of relative integrals} We finish the section with a discussion on the existence of an $A/B$-integral (these results will not be used in the latter sections).
 Let $B \subset A$ be a Hopf subalgebra. We have seen in Proposition \ref{prop:subsmooth} that, under some mild conditions, the existence of a non-trivial right $A/B$-integral implies that $A$ is finitely generated as a $B$-module. The converse is not true, as shown by the following example.

\begin{example}
Let 
$$A= \mathbb C\langle x,g, g^{-1} \ | \ gg^{-1}=1=g^{-1}g,\  x^2=0, \ xg=-gx\rangle$$
with the Hopf algebra structure defined by 
$$\Delta(g)=g \otimes g, \ \Delta(x)=1\otimes x +x\otimes g, \ \varepsilon(g)=1, \ \varepsilon(x)=0, \ S(g)=g^{-1}, \ S(x)=-xg^{-1}$$
and let $B=\mathbb C[g,g^{-1}] \simeq \mathbb C\mathbb Z$. Then $A$ is finitely generated as a $B$-module,  the vector space $A/B^+A$ has dimension $2$ (with $\{p(1),p(x)\}$ as a linear basis, where  $p : A \rightarrow A/B^+A$ is, as before, the canonical surjection), but there does not exist a non-trivial right $A/B$-integral.  
\end{example}

\begin{proof}
 Let $t \in A$ be a right $A/B$-integral, with 
$$t = \sum_{i \in \mathbb Z} \lambda_ig^i + \sum_{j\in \mathbb Z}\mu_jg^jx$$
Since $p(g^i)=p(1)$ and $p(g^jx)=p(x)$, for any $i,j$, we have $p(t)=(\sum_i\lambda_i)p(1)+(\sum_j\mu_j)p(x)$.
Since $t$ is a right $A/B$-integral, we have $tg-t\in B^+A$ and $tx\in B^+A$, hence $p(tg)=p(t)$ and $p(tx)=0$.
We then have 
$$p(tx)=p( \sum_{i} \lambda_ig^ix)= (\sum_{i}\lambda_i)p(x), \ {\rm hence} \ \sum_{i}\lambda_i=0,   $$
$${\rm and} \ p(tg)= p(\sum_{i} \lambda_ig^{i+1} - \sum_{j}\mu_jg^{j+1}x)=(\sum_{i} \lambda_i)p(1)-(\sum_{j}\mu_j)p(x) = -(\sum_{j}\mu_j)p(x),$$
hence $\sum_{j}\mu_j=0$, and $p(t)=0$: our integral is trivial.
\end{proof}

On the positive side, we have the following result, inspired by considerations in \cite[Section 3]{chi}.

\begin{proposition}\label{prop:exisint}
Let $B \subset A$ be a Hopf subalgebra, with $A$ finitely generated as a left $B$-module.
 Assume that one of the following conditions holds.
\begin{enumerate}
 \item $A$ is of Kac type, i.e.\ the square of the antipode is the identity.
\item $A$ is a compact Hopf algebra, i.e.\ $A$ is the canonical dense Hopf $*$-algebra associated to a compact quantum group.
\end{enumerate}
Then there exists an  $A/B$-integral $t$ with $\varepsilon(t)\not=0$.
\end{proposition}

\begin{proof}
 First notice that $S^2$ induces an endomorphism of $C=A/B^+A$, that we denote $S^2_C$. Since $A$ is finitely generated as a left $B$-module, the coalgebra $C$ is finite-dimensional. Let $x_i$, $i \in I$, be a (finite) basis of $C$, with dual basis $x_i^*$. A natural candidate to provide a non-trivial $A/B$-integral is then
$$\tau=\sum_{i \in I} x_i^*\left(S^2_C(x_{i(1)})\right)x_{i(2)}\in C$$
Indeed, one checks, very similarly to the Hopf algebra case (see \cite[Proposition 1.1]{vandae}), that $\tau\cdot a = \varepsilon(a)\tau$ for any $a \in A$ (recall that the right $A$-module structure on $C$ is given by $p(a)\cdot a'=p(aa')$). Hence choosing $t \in A$ such that $p(t)=\tau$ provides an $A/B$-integral.  We then have 
$$\varepsilon(t)=\varepsilon(\tau)={\rm Tr}(S^2_C)$$ 
If $S^2={\rm id}$, then $S^2_C={\rm id}$, and ${\rm Tr}(S^2_C)=\dim(C)\not=0$ (since $C\not=0$), so we have the announced result. In the compact case, recall, see e.g. \cite[Chapter 11]{ks}, that $S^2 \colon A \rightarrow A$ preserves every subcoalgebra of $A$, and its restriction to such a coalgebra is a diagonalizable endomorphism with positive eigenvalues. Since  a finite-dimensional subspace of a coalgebra is contained in a finite-dimensional subcoalgebra, there exists a finite-dimensional subcoalgebra $D \subset A$ such that $p_{|D} \colon D \rightarrow C$ is surjective. From the above properties of $S^2$, we see that ${\rm Tr}(S^2_C)>0$, which concludes the proof.
\end{proof}

\begin{corollary}\label{coro:cdfiniteindex}
 Let $B \subset A$ be a Hopf subalgebra, with $A$ finitely generated as a left $B$-module.
 Assume that one of the following conditions holds.
\begin{enumerate}
\item $A$ is of Kac type, and faithfully flat as a left or right $B$-module. 
\item $A$ is cosemisimple of Kac type.
\item $A$ is a compact Hopf algebra.
\end{enumerate}
Then 
${\rm cd}(A)={\rm cd}(B)$.
\end{corollary}

\begin{proof}
 The proof follows by combining Propositions \ref{prop:subsmooth} and \ref{prop:exisint}, together with the fact that a cosemisimple Hopf algebra is faithfully flat over its Hopf subalgebras \cite[Theorem 2.1]{chi}.
\end{proof}

We believe that the conclusion of Corollary \ref{coro:cdfiniteindex} holds whenever $A$ is cosemisimple, but we have no proof of this.

%\begin{theorem}
% Let $B \subset A$ be a Hopf subalgebra. Assume that $A$ is faithfully flat as a (?) $B$-module, and that there exists a right $A/B$-integral $t \in A$ with $\varepsilon(t)\not=0$.
%\begin{enumerate}
% \item We have ${\rm cd}(A)={\rm cd}(B)$.
%\item For any $M \in \mathcal M_A$ and any $N \in \mathcal M_A^{C}$, where $C=A/B^+A$, we have
%$${\rm Ext}_A^*(M,N) \simeq {\rm Ext}_B^*(M_{|B}, N^{{\rm co} C})$$
%\item If $A$ is a twisted Calabi-Yau algebra of dimension $d$, then so is $B$.
%\end{enumerate}
%\end{theorem}

\section{Quantum symmetry algebras}

We now will apply Corollary \ref{coro:pslCY}, combined with crucial results from \cite{mro} and  \cite{wyz}, to show that $A_s(n)$ is Calabi-Yau of dimension $3$ when $n\geq 4$. In fact our result concerns a more general family of algebras, that we recall now.

Let $(R,\varphi)$ be a \textsl{finite-dimensional measured algebra}, i.e.\  $R$ is a finite-dimensional algebra and $\varphi \colon R \rightarrow \C$ is a linear map (a measure on $R$) such that the associated bilinear map $R \times R \rightarrow \C$, $(x,y) \mapsto \varphi(xy)$ is non-degenerate. Thus a  finite-dimensional measured algebra is a  Frobenius algebra together with a fixed measure.
A coaction of a Hopf algebra $A$ on a finite-dimensional measured algebra $(R,\varphi)$ is an $A$-comodule algebra structure on $R$ such that $\varphi \colon R \rightarrow \C$ is $A$-colinear.

\begin{definition}
 The universal Hopf algebra coacting on a finite-dimensional measured algebra $(R,\varphi)$ is denoted $A_{\rm aut}(R,\varphi)$, and is called the \textsl{quantum symmetry algebra of $(R,\varphi)$}.
\end{definition}

The existence of $A_{\rm aut}(R,\varphi)$ (see \cite{wan98} in the compact case with $R$ semisimple and \cite{bi00} in general) follows from standard Tannaka-Krein duality arguments. 
The following particular cases are of special interest.

\begin{enumerate}
 \item For $R=\C^n$ and $\varphi=\varphi_n$ the canonical integration map (with $\varphi_n(e_i)=1$ for $e_1, \ldots , e_n$ the canonical basis of $\C^n$), we have  $A_{\rm aut}(\C^n,\varphi_n)=A_s(n)$, the coordinate algebra of the quantum permutation group \cite{wan98}, discussed in the introduction.
\item For $R=M_2(\C)$ and $q \in \C^*$, let ${\rm tr}_q \colon M_2(\C) \rightarrow  \C$ be the $q$-trace, i.e.\ for $g=(g_{ij})\in M_2(\C)$,  ${\rm tr}_q(g) =qg_{11}+q^{-1}g_{22}$ . Then we have $A_{\rm aut}(M_2(\C),{\rm tr}_q) \simeq \mathcal O({\rm PSL}_q(2))$, the  isomorphism  being constructed using the universal property  of $A_{\rm aut}(M_2(\C),{\rm tr}_q)$ (and the verification that it is indeed an isomorphism being a long and tedious computation, as in \cite{dij}).
\end{enumerate}

Let $(R,\varphi)$ be a finite-dimensional measured algebra. Since $\varphi \circ m$ is non-degenerate, where $m$ is the multiplication of $R$, there exists a linear map $\delta \colon \C \rightarrow R \otimes R$ such that 
$(R,\varphi \circ  m, \delta)$ is a left dual for $R$, i.e.\
$$((\varphi \circ m)\otimes {\rm id}_R) \circ ({\rm id}_R \otimes \delta)={\rm id}_R= 
({\rm id}_R \otimes  (\varphi \circ m)) \circ (\delta \otimes {\rm id}_R)$$
 Following \cite{mro}, we put
$$\tilde{\varphi}= \varphi  \circ m \circ (m \otimes {\rm id}_R) \circ ({\rm id}_R \otimes  \delta)\colon R \rightarrow \C$$
 Using the definition of Frobenius algebra in terms of coalgebras, 
the coproduct of $R$ is $\Delta = (m \otimes {\rm id}_R) \circ ({\rm id}_R \otimes  \delta)=({\rm id}_R\otimes m) \circ ( \delta \otimes  {\rm id}_R)$, and we have $\tilde{\varphi}=\varphi \circ m \circ \Delta$.

\begin{definition} We say that $(R,\varphi)$ (or $\varphi$) is \textsl{normalizable} if $\varphi(1)\not= 0$ and if there exists $\lambda \in \C^*$ such that $\tilde{\varphi}=\lambda \varphi$. 
\end{definition}

The condition that $\varphi$ is normalizable is equivalent to require, in the language of  \cite[Definition 3.1]{kad}, that $R/\C$ is a strongly separable extension with Frobenius system $(\varphi, x_i,y_i)$, where $\delta(1)=\sum_ix_i \otimes y_i$. Hence it follows that if $\varphi$ is normalizable, then $R$ is necessarily a separable (semisimple) algebra. Conversely, if $R$ is semisimple,  writing $R$ as a direct product of matrix algebras, one easily sees the conditions that ensure that $\varphi$ is normalizable, see \cite{mro}.

The following result is 
\cite[Corollary 4.9]{mro}, generalizing earlier results from \cite{ba99,ba02,dervan}:

\begin{theorem}
Let $(R,\varphi)$ is a finite-dimensional semisimple measured algebra with $\dim(R)\geq 4$ and $\varphi$ normalizable. Then there exists an equivalence of $\mathbb C$-linear tensor categories   
$$\mathcal M^{A_{\rm aut}(R,\varphi)} \simeq^{\otimes} \mathcal M^{\mathcal O({\rm PSL}_q(2))}$$
for some $q \in \C^*$ with $q+q^{-1} \not=0$.
\end{theorem}

The above parameter $q$ is determined as follows. First consider $\lambda \in \C^*$ such that $\tilde{\varphi}=\lambda \varphi$ and choose $\mu \in \C^*$ such that $\mu^2 = \lambda \varphi(1)$. Then $q$ is any solution of the equation
$q+q^{-1}=\mu$ (recall that $\mathcal O({\rm PSL}_q(2))=\mathcal O({\rm PSL}_{-q}(2))$, so the choice of $\mu$ does not play any role).

As an example, for $(\C^n,\varphi_n)$ as above (and $n \geq 4$), it is immediate that $\varphi_n$ is normalizable with the corresponding $\lambda$ equal to $1$, and $q$ is any solution of the equation $q+q^{-1}=\sqrt{n}$. 

It was shown in \cite{bic} (Theorem 6.5 and the comments after) that if $(R,\varphi)$ is a finite-dimensional semisimple measured algebra with $\dim(R)\geq 4$ and  $\varphi$ normalizable and such that $A_{\rm aut}(R,\varphi))$ is cosemisimple, then ${\rm cd}(A_{\rm aut}(R,\varphi))\leq 3$, with equality if $\varphi$ is a trace. We generalize this result here, as follows.

\begin{theorem}\label{thm:Aaut}
 Let $(R,\varphi)$ be a finite-dimensional semisimple measured algebra with $\dim(R)\geq 4$ and $\varphi$ normalizable. Then $A_{\rm aut}(R,\varphi)$ is a twisted Calabi-Yau algebra of dimension $3$, and is Calabi-Yau if $\varphi$ is a trace.
\end{theorem}

\begin{proof}
We have already pointed out that, by \cite[Corollary 4.9]{mro}, there exists $q \in \C^*$, with $q+q^{-1} \not=0$, such that 
$$\mathcal M^{A_{\rm aut}(R,\varphi)} \simeq^{\otimes} \mathcal M^{\mathcal O({\rm PSL}_q(2))}$$
Put, for notational simplicity, $A= A_{\rm aut}(R,\varphi)$. 
By the proof of Theorem 6.4 in \cite{bic} and the comments after \cite[Theorem 6.5]{bic}, the trivial Yetter-Drinfeld module over $A$ has a finite resolution by finitely generated relative projective Yetter-Drinfeld modules over $A$, so since $\mathcal O({\rm PSL}_q(2))$ is twisted Calabi-Yau of dimension $3$ by Corollary \ref{coro:pslCY}, it follows from \cite[Theorem 4]{wyz} that $A$ is twisted Calabi-Yau of dimension $3$.

If $\varphi$ is a trace, it follows from the normalizability and from the analysis in Section $2$ of \cite{mro} that $\mu$ as above is a real number such that $\mu^2 \geq 4$, so $q$ as above is such that  $\mathcal O({\rm PSL}_q(2))$ is cosemisimple, and hence $A$ is cosemisimple. Thus by \cite[Theorem 6.5]{bic} we have $H^3_b(A)\simeq \mathbb C$, where $H_b^*$ denotes bialgebra cohomology \cite{gs1}, so 
by \cite[Proposition 5.9]{bic} we have $H^3(A, \mathbb C) \simeq {\rm Ext}_A^3(\mathbb C_\varepsilon, \mathbb C_\varepsilon) \not =(0)$. 
We conclude that $A$ is Calabi-Yau by Lemma \ref{lemm:tcy-cy}.
%Let $\xi : A \rightarrow \mathbb C$ be an algebra map such that ${\rm Ext}_A^3(\mathbb C, A)\simeq {_\xi \! \mathbb C}$. 
%By homological duality, we have 
%$${\rm Ext}^3_A(\mathbb C, \mathbb C) \simeq {\rm Tor}_0^A(\mathbb C, {_\xi \!\mathbb C})
%\simeq \begin{cases}
%  0 & \text{if $\xi \not= \varepsilon$} \\
%\mathbb C & \text{if $\xi = \varepsilon$}
% \end{cases}
%$$
%Hence $\xi = \varepsilon$, and $A$ is Calabi-Yau, as claimed.
\end{proof}

\begin{proof}[Proof of Theorem \ref{thm:asncala}] We already know from the previous result that $A_s(n)$ is Calabi-Yau of dimension $3$ for $n \geq 4$, and it remains to compute the cohomology spaces with trivial coefficients. We have, as usual, $H^0(A_s(n), \mathbb C)\simeq \mathbb C$, so the homological duality gives $H^3(A_s(n), \mathbb C) \simeq \mathbb C$. It is easy to check, using the fact that $A_s(n)$ is generated by projections, that $H^1(A_s(n), \mathbb C)=(0)$. We also have 
$$H^2(A_s(n), \mathbb C) \simeq H_1(A_s(n), \mathbb C) \simeq {\rm Tor}_1^{A_s(n)}(\mathbb C_\varepsilon, {_\varepsilon \! \mathbb C})$$
where the latter space is the quotient of $A_s(n)$ by the subspace generated by the elements $ab-\varepsilon(a)b-\varepsilon(b)a$, $a,b \in A_s(n)$, and is easily seen to be zero.  
\end{proof}

\section{Second cohomology  of a quantum permutation algebra}

In this section we show that the vanishing of the second cohomology with trivial coefficients of $A_s(n)$ holds for a much more general class of quantum permutation algebras.

Given a matrix $d=(d_{ij}) \in M_n(\mathbb C)$, we denote by $A_s(n,d)$ the quotient of $A_s(n)$ by the relations
$$\sum_{k=1}^n d_{ik}u_{kj} = \sum_{k=1}^n u_{ik}d_{kj}$$
 The algebras $A_s(n,d)$ are Hopf algebras, with structure maps induced by those of $A_s(n)$.
When $d$ is the adjacency matrix of a finite graph $X$, the Hopf algebra $A_s(n,d)$ represents the quantum automorphism group of the graph $X$, in the sense of \cite{ban}. 

\begin{example}
As interesting examples of Hopf algebras of the type $A_s(n,d)$, let us mention the coordinate algebras of the so-called quantum reflection groups \cite{bic04, bv}. Let $p \geq 1$ and let $A_h^p(n)$ be the algebra presented by generators $u_{ij}$, $1\leq i,j \leq n$, subject to the relations
$$u_{ij}u_{ik}=0= u_{ji}u_{ki}, \ k \not=j, \ \sum_{j=1}^nu_{ij}^p=1=\sum_{j=1}^nu_{ji}^p$$
This is a Hopf algebra again \cite{bic04}, with $S(u_{ij})=u_{ji}^{p-1}$, and we have $A_h^p(n) \simeq A_s(np, d)$, where $d$ is the adjacency matrix of the graph formed by $n$ disjoint copies of an oriented $p$-cycle \cite[Corollary 7.6]{bb}.
\end{example}

%\subsection{Main result}
%If $(A,\varepsilon)$ is an augmented algebra, we denote by $H^*(A, \mathbb C)$ the Hochschild cohomology of $A$ with coefficients in the trivial bimodule ${_\varepsilon\!\mathbb C}_\varepsilon$. The main result of this note is the following one.

\begin{theorem}\label{thm:vanish}
 We have $H^2(A_s(n,d), \mathbb C)=(0)$ for any matrix $d \in M_n(\mathbb C)$.
\end{theorem}

\begin{rems} 
 \begin{enumerate}
\item It is an immediate verification that $H^1(A_s(n,d), \mathbb C)=(0)$. On the other hand, it is not possible to extend the vanishing result of Theorem \ref{thm:vanish} to degree $3$, since $H^3(A_s(n), \mathbb C)\not=(0)$ for $n \geq 4$, by Theorem \ref{thm:asncala}.
\item
 It follows from Theorem \ref{thm:vanish} and Proposition 5.9 in \cite{bic} that $H^2_b(A_s(n,d))=(0)$ as well, where $H^*_b$ denotes the bialgebra cohomology of Gerstenhaber-Schack \cite{gs1}. 
\item
 The recent $L^2$-Betti numbers computations in \cite{krvv} show that the algebra $A_h^p(n)$ is not Calabi-Yau for $p \geq 2$, hence Theorem \ref{thm:asncala} cannot be generalized to the class of Hopf algebras  $A_s(n,d)$.
\end{enumerate}
\end{rems}

%\begin{remark}{\rm
% Il follows from Theorem \ref{thm:vanish} and \cite{fgt} that the algebras $A_s(n,d)$ have the (AC) property: all 
%cocycles can be completed to a Sch\"urmann triple. See \cite{fgt} for details.
%}
%\end{remark}

\subsection{A general lemma}

In this subsection we present a general lemma, useful to show when a $2$-cocycle of an augmented algebra is a coboundary.

We fix an augmented algebra $(A, \varepsilon)$, and put as usual $A^+={\rm Ker}(\varepsilon)$. We denote by $T_3(A)$ the subspace of endomorphisms of $\mathbb C \oplus A^+ \oplus \mathbb C$ represented by matrices
$$\begin{pmatrix}
  \lambda & f & \mu \\
0 & a & b \\
0 & 0  & \gamma
  \end{pmatrix}
$$
where $\lambda, \mu, \gamma \in \mathbb C$, $f \in (A^+)^*$, $a \in A$, $b \in A^+$,
and such a matrix acts on $(\alpha, c ,\beta) \in \mathbb C \oplus A^+ \oplus \mathbb C$ by
$$\begin{pmatrix}
  \lambda & f & \mu \\
0 & a & b \\
0 & 0  & \gamma
  \end{pmatrix} \begin{pmatrix}
\alpha \\
c \\
\beta
\end{pmatrix} =
\begin{pmatrix}
 \lambda \alpha +f(c) + \mu \beta\\
ac + \beta b \\
\gamma \beta
\end{pmatrix}
$$
It is straightforward to check that the composition of two such operators is again of the same type, with the composition corresponding to the naive matrix multiplication
$$\begin{pmatrix}
  \lambda & f & \mu \\
0 & a & b \\
0 & 0  & \gamma
  \end{pmatrix} \begin{pmatrix}
  \lambda' & f' & \mu' \\
0 & a' & b' \\
0 & 0  & \gamma'
  \end{pmatrix} =
\begin{pmatrix}
  \lambda \lambda'& \lambda f' +f(a'-) & \lambda \mu' +f(b')+\mu \gamma' \\
0 & aa' & ab' +b\gamma'\\
0 & 0  & \gamma \gamma'
  \end{pmatrix}
$$
and thus $T_3(A)$ is a subalgebra of ${\rm End}(\mathbb C \oplus A^+ \oplus \mathbb C)$.

Now let $ \psi \colon A \rightarrow \mathbb C$ and $c \colon A \otimes A \rightarrow \mathbb C$ be normalized linear functionals, which means that $\psi(1)=0$ and $c(1 \otimes 1)=0$. Recall (see Section 2) that $c$ is a Hochschild $2$-cocycle for the trivial bimodule $\mathbb C$ if and only if for any $a,b \in A$, we have
$$\varepsilon(a)c(b\otimes -) -c(ab\otimes -) + c(a\otimes b-)-c(a\otimes b)\varepsilon(-)=0 \quad (\star)$$ It is easy to see that a normalized $2$-cocycle satisfies 
$c(a \otimes 1)=0=c(1 \otimes a)$ for any $a \in A$.

\begin{lemma}\label{useful}
Let $ \psi \colon A \rightarrow \mathbb C$ be a normalized linear functional, and let $c \in Z^2(A, \mathbb C)$ be a normalized $2$-cocycle.
 The map 
\begin{align*}
 \rho \colon A & \longrightarrow T_3(A) \\
a  & \longmapsto \begin{pmatrix}
  \varepsilon(a) & c(a\otimes -) & \psi(a) \\
0 & a & a-\varepsilon(a) \\
0 & 0  & \varepsilon(a)
  \end{pmatrix}
\end{align*}
is an algebra map if and only if $c = \delta(-\psi)$
\end{lemma}

\begin{proof}
We have $\rho(1)=1$ since $c$ and $\psi$ are normalized. For $a, b \in A$,  the $2$-cocycle condition gives
$$c(ab\otimes -)_{|A^+}= (\varepsilon(a)c(b\otimes -)+c(a\otimes b-))_{|A^+}$$
It thus follows that
 $\rho(ab)=\rho(a)\rho(b)$ if and only if
$$c(a \otimes (b-\varepsilon(b)))= c(a\otimes b)=-\varepsilon(a) \psi(b) +\psi(ab)-\psi(a)\varepsilon(b)=\delta(-\psi)(a \otimes b)$$ 
and this gives the announced result.
\end{proof}

\subsection{Proof of Theorem \ref{thm:vanish}}

We now let $A=A_s(n,d)$ as in the beginning of the section. The following lemma provides the main properties of $2$-cocycles on $A$.

\begin{lemma}\label{comput}
 Let $c \in Z^2(A, \mathbb C)$ be a normalized $2$-cocycle. 
\begin{enumerate}
 \item For $j\not=k$, we have
$$\delta_{ij}c(u_{ik}\otimes -)+ c(u_{ij}\otimes u_{ik}-) - c(u_{ij}\otimes u_{ik}) \varepsilon(-)=0$$
\item For $j\not=k$, we have
$$\delta_{ij}c(u_{ki}\otimes -)+ c(u_{ji}\otimes u_{ki}-) - c(u_{ji}\otimes u_{ki}) \varepsilon(-)=0$$
\item For $i \not=j$, $k \not=i$, $k \not=j$, we have $c(u_{ij} \otimes u_{ik})=0=c(u_{ji}\otimes u_{ki})$.
\item We have $c(u_{ij} \otimes u_{ii})=c(u_{ii} \otimes u_{ij})=c(u_{ij} \otimes u_{jj})=c(u_{jj} \otimes u_{ij})$.
\item We have 
$$\delta_{ij}c(u_{ij} \otimes -) - c(u_{ij} \otimes -) + c(u_{ij} \otimes u_{ij}-) -\varepsilon(-)c(u_{ij} \otimes u_{ij})=0$$
\item $c(u_{ii} \otimes u_{ii}) = \frac{1}{2}\sum_j c(u_{ij} \otimes u_{ij}) = \frac{1}{2}\sum_j c(u_{ji} \otimes u_{ji})$.
\item We have for $i \not=j$, $c(u_{ji}\otimes u_{ji})= -c(u_{ii} \otimes u_{ji})$.
\item Put $X_{ij}=\sum_kd_{ik}u_{kj} = \sum_k u_{ik}d_{kj}$. Then 
$$c(u_{ii} \otimes X_{ij}) = c(X_{ij} \otimes u_{ii})= c(X_{ij} \otimes u_{jj})=c(u_{jj} \otimes X_{ij})$$
\end{enumerate}
\end{lemma}

\begin{proof}
 (1) (resp. (2)) follows from the cocycle identity $(\star)$ applied to $a=u_{ij}$ and $b=u_{ik}$ (resp. to $a=u_{ji}$ and $b=u_{ki}$). The first (resp. second) identity in (3) follows from (1) (resp. (2)) applied to $u_{ii}$.

In (1), for $i=j \not=k$, we get 
$$c(u_{ik}\otimes -)  + c( u_{ii}\otimes u_{ik}-)-c(u_{ii} \otimes  u_{ik})\varepsilon(-)=0$$ which gives, applied to $u_{ii}$ and to $u_{kk}$, 
$$c(u_{ii}\otimes u_{ik})=c(u_{ik} \otimes u_{ii}), \  c(u_{ik}\otimes u_{kk})=c(u_{ii} \otimes u_{ik})$$
We get the remaining relation in (4) using (2) at $i=j$. (5) follows from the cocycle identity $(\star)$ applied to $a=u_{ij}=b$, and to 
 get (6), we sum (5) over $j$ to get
$$c(u_{ii}\otimes -) + \sum_j c(u_{ij}\otimes u_{ij}-)-\sum_jc(u_{ij} \otimes u_{ij})\varepsilon(-)=0$$
which gives in particular
$$c(u_{ii}\otimes u_{ii}) +  c(u_{ii}\otimes u_{ii}^2) = 2c(u_{ii}\otimes u_{ii})= \sum_jc(u_{ij} \otimes u_{ij})$$
and the proof that $$c(u_{ii}\otimes u_{ii}) = \frac{1}{2}\sum_j c(u_{ji} \otimes u_{ji})$$
is similar. (7) follows from (1) at $i=j\not=k$, applied at $u_{ik}$.
To prove (8), first notice that
$$u_{ii}X_{ij}= d_{ij}u_{ii}=X_{ij}u_{ii}, \ u_{jj}X_{ij}= d_{ij}u_{jj}=X_{ij}u_{jj}$$
The cocycle identity $(\star)$, applied to $a=u_{ii}$, $b=X_{ij}$, gives
$$c(X_{ij}\otimes -) - c(u_{ii} X_{ij} \otimes -)+ c(u_{ii} \otimes X_{ij}-) - c(u_{ii} \otimes X_{ij})\varepsilon(-)=0$$
 and hence
$$c(X_{ij}\otimes u_{jj}) - c(u_{ii} X_{ij} \otimes u_{jj})+ c(u_{ii} \otimes X_{ij}u_{jj}) - c(u_{ii} \otimes X_{ij})=0$$
which gives $c(X_{ij}\otimes u_{jj})  = c(u_{ii} \otimes X_{ij})$. Using (4), we also have
$$c(X_{ij} \otimes u_{jj})= \sum_kd_{ik}c(u_{kj} \otimes u_{jj})=\sum_kd_{ik}c(u_{jj} \otimes u_{kj})=
c(u_{jj} \otimes X_{ij})$$
and similarly
$$c(X_{ij}\otimes u_{ii})=c(u_{ii}\otimes X_{ij})$$
This concludes the proof of the lemma.
\end{proof}

\begin{lemma} \label{rep}
 Let $c \in Z^2(A, \mathbb C)$ be a normalized $2$-cocycle. Put, for any $i,j$,
$\lambda_{ij}=-c(u_{ii} \otimes u_{ij})$. Then there exists an algebra map
$\rho_c \colon A \rightarrow T_3(A)$ such that
$$\rho_c(u_{ij}) = \begin{pmatrix}
  \delta_{ij} & c(u_{ij}\otimes -) & \lambda_{ij} \\
0 & u_{ij} & u_{ij}-\delta_{ij} \\
0 & 0  & \delta_{ij}
  \end{pmatrix}$$
\end{lemma}

\begin{proof}
 Let $t_{ij}$ be the element of $T_3(A)$ corresponding to the above matrix. We have to check that the elements $t_{ij}$ satisfy the defining relations of $A$. Using the fact that $c$ is normalized, we easily see that $\sum_{j=1}^nt_{ij}=1$, and similarly, using Relations (4) in Lemma \ref{comput}, we see that $\sum_{j=1}^n t_{ji}=1$. We have
\begin{align*}t_{ij}t_{ik}& = \begin{pmatrix}
  \delta_{ij} \delta_{ik}& \delta_{ij}c(u_{ik}\otimes -) +c(u_{ij}\otimes u_{ik}-)& \delta_{ij}\lambda_{ik} +c(u_{ij}\otimes (u_{ik}-\delta_{ik}))+ \lambda_{ij}\delta_{ik}\\
0 & u_{ij} u_{ik}& u_{ij}u_{ik}-\delta_{ik}u_{ij}+\delta_{ik}u_{ij}-\delta_{ij}\delta_{ik} \\
0 & 0  & \delta_{ij}\delta_{ik}
  \end{pmatrix} \\
&= \begin{pmatrix}
  \delta_{ij} \delta_{ik}& \delta_{ij}c(u_{ik}\otimes -) +c(u_{ij}\otimes u_{ik}-)& \delta_{ij}\lambda_{ik} +c(u_{ij}\otimes u_{ik})+ \lambda_{ij}\delta_{ik}\\
0 & u_{ij} u_{ik}& u_{ij}u_{ik}-\delta_{ij}\delta_{ik} \\
0 & 0  & \delta_{ij}\delta_{ik}
  \end{pmatrix}
\end{align*}
Hence for $j \not=k$ we have
$$t_{ij}t_{ik} = \begin{pmatrix}
  0& \delta_{ij}c(u_{ik}\otimes -) +c(u_{ij}\otimes u_{ik}-)& -\delta_{ij}c(u_{ii} \otimes u_{ik}) +c(u_{ij}\otimes u_{ik})- c(u_{ii} \otimes u_{ij})\delta_{ik}\\
0 & 0& 0\\
0 & 0  & 0
  \end{pmatrix}$$
By (1) in Lemma \ref{comput}, we have $(\delta_{ij}c(u_{ik}\otimes -) +c(u_{ij}\otimes u_{ik}-))_{|A^+}=0$. Moreover  by (3) and (4) in Lemma \ref{comput} we have
$-\delta_{ij}c(u_{ii} \otimes u_{ik}) +c(u_{ij}\otimes u_{ik})- c(u_{ii} \otimes u_{ij})\delta_{ik}=0$. Hence $t_{ij}t_{ik}=0$ for $j \not =k$. We have also
$$t_{ij}^2 = \begin{pmatrix}
  \delta_{ij} & \delta_{ij}c(u_{ij}\otimes -) +c(u_{ij}\otimes u_{ij}-)& 2\delta_{ij}\lambda_{ij} +c(u_{ij}\otimes u_{ij})\\
0 & u_{ij}& u_{ij}-\delta_{ij} \\
0 & 0  & \delta_{ij}
  \end{pmatrix}$$
We have 
$$(\delta_{ij}c(u_{ij}\otimes -) +c(u_{ij}\otimes u_{ij}-))_{|A^+} = c(u_{ij}\otimes -)_{|A^+}$$
by (5) in Lemma \ref{comput}, and $-2\delta_{ij}c(u_{ii} \otimes u_{ij}) +c(u_{ij}\otimes u_{ij})=-c(u_{ii} \otimes u_{ij})$ by (7) and (4) in Lemma \ref{comput}, hence $t_{ij}^2=t_{ij}$.

The proof that for $t_{ji}t_{ki}=0$ for $j \not=k$ is similar, using (2), (3) and (4). Finally that $\sum_kd_{ik}t_{kj} = \sum_k t_{ik}d_{kj}$ follows from (4) and (8) in Lemma \ref{comput}.
\end{proof}

We are now ready to prove Theorem \ref{thm:vanish}. Let $c \in Z^2(A, \mathbb C)$ be a Hochschild $2$-cocycle, that we can assume to be normalized without changing its cohomology class.
 
A representation $\rho \colon A \rightarrow T_3(A)$  is necessarily of the form  
$$a \longmapsto \begin{pmatrix}
  \varepsilon_1(a) & C(a) & \psi(a) \\
0 & f(a) & g(a) \\
0 & 0  & \varepsilon_2(a)
  \end{pmatrix}$$
for algebra maps $$\varepsilon_1, \ \varepsilon_2 \colon A \rightarrow \mathbb C, \  f \colon A \rightarrow A$$ and linear  maps
$$C \colon A \rightarrow (A^+)^*, \ g \colon A \rightarrow A^+, \ \psi \colon A \rightarrow \mathbb C$$
Using the definition of the representation $\rho_c$ of Lemma \ref{rep}, we see that for any $a \in A$ we have
$$\rho_c(a)= \begin{pmatrix}
  \varepsilon(a) & C(a) & \psi(a) \\
0 & a & a -\varepsilon(a) \\
0 & 0  & \varepsilon(a)
  \end{pmatrix}$$
with moreover, for any $a,b$, $C(ab)= \varepsilon(a)C(b) + C(a)(b-)$. Since $C(u_{ij})=c(u_{ij}\otimes -)$, an easy induction shows that $C(a)= c(a \otimes -)$ for any $a$. We conclude from Lemma \ref{useful} that $c$ is a coboundary, as needed.


\begin{thebibliography}{}

\bibitem{ad}
N.~Andruskiewitsch, J.~Devoto, 
Extensions of Hopf algebras, 
 \emph{St.\ Petersburg Math.\ J.}  {\bf 7} (1996), no. 1, 17-52.

\bibitem{ba99} T. Banica, Symmetries of a generic coaction, \emph{Math. Ann.} {\bf 314} (1999), no. 4, 763-780.

\bibitem{ba02} T. Banica, Quantum groups and Fuss-Catalan algebras, \emph{Comm. Math. Phys.} {\bf 226} (2002), no. 1, 221-232. 

\bibitem{ban} T. Banica,  Quantum automorphism groups of homogeneous graphs, \emph{J. Funct. Anal.} {\bf 224} (2005), no. 2, 243-280. 

\bibitem{bb} T. Banica, J. Bichon, Free product formulae for quantum permutation groups, \emph{J. Inst. Math. Jussieu} {\bf 6} (2007), no. 3, 381-414.

\bibitem{bv} T. Banica, R. Vergnioux, Fusion rules for quantum reflection groups, \emph{J. Noncommut. Geom.} 3 (2009), no. 3, 327-359.



\bibitem{bi00} J. Bichon,  Galois reconstruction of finite quantum groups,  \emph{J. Algebra} {\bf 230} (2000), no. 2, 683-693. 

\bibitem{bic04}  J. Bichon, Free wreath product by the quantum permutation group, \emph{Algebr. Represent. Theory} {\bf 7} (2004), no. 4, 343-362. 

\bibitem{bi08} J. Bichon,  Algebraic quantum permutation groups,  \emph{Asian-Eur. J. Math.} {\bf 1} (2008), no. 1, 1-13. 

\bibitem{bic13}
 J. Bichon, Hochschild homology of Hopf algebras and free Yetter-Drinfeld resolutions of the counit, \emph{Compos. Math.} {\bf 149} (2013), no. 4, 658-678.


 \bibitem{bic}
J. Bichon, Gerstenhaber-Schack and Hochschild cohomologies of Hopf algebras, \emph{Doc. Math.} {\bf 21} (2016), 955-986. 
 

\bibitem{biec} R. Bieri, B. Eckmann, Groups with homological duality generalizing Poincar\'e duality, \emph{Invent. Math.} {\bf 20} (1973), 103-124. 

\bibitem{bra} M. Brannan,  Reduced operator algebras of trace-perserving quantum automorphism groups, \emph{Doc. Math.} {\bf 18} (2013), 1349-1402.

\bibitem{brown}
K.S. Brown, Cohomology of groups, Graduate Texts in Mathematics 87, Springer-Verlag, 1982.


\bibitem{bz} K.A. Brown, J.J. Zhang, Dualising complexes and twisted Hochschild (co)homology for Noetherian Hopf algebras, \emph{J. Algebra}  {\bf 320}  (2008),  no. 5,  1814-1850. 

\bibitem{chi} A. Chirvasitu, Cosemisimple Hopf algebras are faithfully flat over Hopf subalgebras, \emph{Algebra Number Theory} {\bf 8} (2014), no. 5, 1179-1199.


\bibitem{cht} B. Collins, J. H\"{a}rtel, A. Thom,  Homology of free quantum groups.  \emph{C. R. Math. Acad. Sci. Paris} {\bf  347}  (2009),  271-276.


\bibitem{dervan} A. De Rijdt, N. Vander Vennet,  Actions of monoidally equivalent compact quantum groups and applications to probabilistic boundaries, \emph{Ann. Inst. Fourier} {\bf 60} (2010), no. 1, 169-216.

\bibitem{dij} M.S. Dijkhuizen, 
The double covering of the quantum group 
$SO_q(3)$, \emph{Rend. Circ. Math. Palermo Serie II} {\bf 37} (1994), 47-57.

%\bibitem{dvl} M. Dubois-Violette, G. Launer,
%The quantum group of a non-degenerate bilinear form, \emph{Phys. Lett. B}
%{\bf 245}, No.2 (1990), 175-177.

\bibitem{fgt} U. Franz, M. Gerhold, A. Thom, On the L\'evy-Khinchin decomposition of generating functionals,  \emph{Commun.  Stoch. Anal.} {\bf 9}, No. 4 (2015) 529-544.

\bibitem{fks} U. Franz, A. Kula, A. Skalski,  L\'evy processes on quantum permutation groups. Noncommutative analysis, operator theory and applications, 193-259, Oper. Theory Adv. Appl., 252, Linear Oper. Linear Syst., Birkh\"auser/Springer,  2016.

%\bibitem{bz} K.A. Brown, J.J. Zhang,  Dualising complexes and twisted Hochschild (co)homology for Noetherian Hopf algebras, \emph{J. Algebra} {\bf 320} (2008), no. 5, 1814-1850.


\bibitem{gs1}
M. Gerstenhaber, S. Schack, Bialgebra cohomology, deformations and quantum groups, \emph{Proc. Nat. Acad. Sci. USA} {\bf 87} (1990), no. 1, 78-81. 

\bibitem{gin} V. Ginzburg, Calabi-Yau algebras, Preprint arXiv:math/0612139.

\bibitem{kad} L.  Kadison, New examples of Frobenius extensions. University Lecture Series 14, American Mathematical Society, Providence, RI, 1999. 

\bibitem{kas}
C. Kassel, Quantum groups,  Graduate Texts in Mathematics 155, Springer, 1995.

\bibitem{ks} A. Klimyk, K. Schm\"{u}dgen,
Quantum groups and their representations, 
Texts and Monographs in Physics,
Springer, 1997. 

\bibitem{krvv}
D. Kyed, S. Raum, S. Vaes, M. Valvekens,
$L^2$-Betti numbers of rigid $C^*$-tensor categories and discrete quantum groups, \emph{ Anal. PDE} {\bf 10} (2017), no. 7, 1757-1791.
 
\bibitem{mon} S. Montgomery, 
Hopf algebras and their actions on rings, Amer. Math. Soc. 1993.

\bibitem{mro} C. Mrozinski, 
Quantum automorphism groups and SO(3)-deformations,
\emph{J. Pure Appl. Algebra} {\bf 219} (2015), no. 1, 1-32. 

\bibitem{musc} E. M\"uller, H.J. Schneider, Quantum homogeneous spaces with faithfully flat module structures, \emph{Israel J. Math.} {\bf 111} (1999), 157-190.

\bibitem{rad} D.E. Radford, Hopf algebras, Series on Knots and Everything, 49. World Scientific, 2012.

\bibitem{sch90} H.-J. Schneider,  Principal homogeneous spaces for arbitrary Hopf algebras, \emph{Israel J. Math.} {\bf 72} (1990), no. 1-2, 167-195.

\bibitem{schn} H.J. Schneider, Normal basis and transitivity of crossed products for Hopf algebras, \emph{ J. Algebra}  {\bf 152}  (1992),  no. 2, 289-312. 

\bibitem{schneider} H.-J. Schneider,
Some remarks on exact sequences of quantum groups,
{\em Comm. Algebra} \textbf{21} (1993), no. 9, 3337-3357.

\bibitem{tak} M. Takeuchi, Relative Hopf modules-equivalences and freeness criteria, \emph{J. Algebra} {\bf 60} (1979),  452-471. 

\bibitem{vandae} Van Daele, A. The Haar measure on finite quantum groups. Proc. Amer. Math. Soc. 125 (1997), no. 12, 3489-3500.

\bibitem{vdb} M. Van den Bergh, A relation between Hochschild homology and cohomology for Gorenstein rings,\emph{ Proc. Amer. Math. Soc.} {\bf 126} (1998) 1345-1348; Erratum: \emph{Proc. Amer. Math. Soc.} {\bf 130} (2002) 2809-2810.

\bibitem{voi} C. Voigt, On the structure of quantum automorphism groups, \emph{J. Reine Angew. Math.}, to appear.

\bibitem{walwan} C. Walton, X. Wang,  On quantum groups associated to non-Noetherian regular algebras of dimension 2, \emph{Math. Z.} {\bf 284} (2016), no. 1-2, 543-574

\bibitem{wan98} S. Wang, Quantum symmetry groups of finite spaces, \emph{Comm. Math. Phys.} {\bf 195} (1998), no. 1, 195-211.

\bibitem{wyz} X. Wang, X. Yu, Y. Zhang, Calabi-Yau property under monoidal Morita-Takeuchi equivalence,  \emph{Pacific J. Math.} {\bf 290} (2017), no. 2, 481-510.  

\bibitem{wei} C. Weibel, An Introduction to Homological Algebra, Cambridge University Press, 1994.

\bibitem{yu} X. Yu,  Hopf-Galois objects of Calabi-Yau Hopf algebras, \emph{J. Algebra Appl.} {\bf 15} (2016), no. 10, 1650194, 19 pp.

\end{thebibliography}
\end{document}